\documentclass[12pt, reqno]{amsproc}
\usepackage{amsmath, amsthm, amscd, amsfonts, amssymb, graphicx}
\usepackage{geometry}
\usepackage{calc}
\usepackage{color,hyperref}
\definecolor{darkblue}{rgb}{0.0,0.0,0.3}
\hypersetup{colorlinks,breaklinks,
            linkcolor=red,urlcolor=pink,
            anchorcolor=darkblue,citecolor=blue}
\newcommand{\makeheading}[2][]%
        {\hspace*{-\marginparsep minus \marginparwidth}%
         \begin{minipage}[t]{\textwidth+\marginparwidth+\marginparsep}%
             {\large \bfseries #2 \hfill #1}\\[-0.15\baselineskip]%
                 \rule{\columnwidth}{1pt}%
         \end{minipage}}

\newtheorem{theorem}{\bf{Theorem}}[section] % 1st argument is your name for it
\newtheorem{lemma}[theorem]{\bf{Lemma}}     % 2nd argument is what is printed
\newtheorem{corollary}[theorem]{\bf{Corollary}}

\newtheorem{rem}[theorem]{\bf{Remark}}
\theoremstyle{definition}

\title[Product-type  operators ]{Product-type  operators between minimal M\"{o}bius invariant  spaces and Zygmund type spaces}
\author[Hassanlu]{Mostafa Hassanlou}
\address{Engineering Faculty of Khoy, Urmia University of Technology, Urmia, Iran}
\email{m.hassanlou@urmia.ac.ir}
\author[Abbasi]{Ebrahim Abbasi$^*$}
\address{Department of Mathematics, Mahabad Branch, Islamic Azad University, Mahabad, Iran}
\email{ebrahimabbasi81@gmail.com
}
\author[Kanani]{Mehdi Kanani
Arpatapeh}
\address{Department of Mathematics, Payame Noor University, P.O. Box 19395-3697 Tehran, Iran}
\email{M.kanani@pnu.ac.ir}
\author[Nasresfahani]{Sepideh Nasresfahani}
\address{Department of Pure Mathematics, Faculty of Mathematics and Statistics, University of Isfahan, Isfahan, Iran}
\email{sepide.nasr@gmail.com
}
\keywords{essential norm, generalized Stevi\'c-Sharma type operators, Zygmund type spaces, minimal M\"{o}bius invariant spaces.\\
\indent $^{*}$ Corresponding author}
\begin{document}
%%%%%%%%%%%%%%%%%%%%%%%%%%%%%%%%%%%%%%%%%%%%%%%%%%%%%%%%%%%%%%%%%%%%%%%%%%%%%%
%%%%%%%%%%%%%%%%%%%%%%%%%%%%%%%%%%%%%%%%%%%%%%%%%%%%%%%%%%%%%%%%%%%%%%%%%%%%%%

%%%%%%%%%%%%%%%%%%%%%%%%%%%%%%%%%%%%%%%%%%%%%%%%%%%%%%%%%%%%%%%%%%%%%%%%%%%%%%
%%%%%%%%%%%%%%%%%%%%%%%%%%%%%%%%%%%%%%%%%%%%%%%%%%%%%%%%%%%%%%%%%%%%%%%%%%%%%%
\vspace{1cm} \setcounter{page}{1} \thispagestyle{empty}
%%%%%%%%%%%%%%%%%%%%%%%%%%%%%%%%%%%%%%%%%%%%%%%%%%%%%%%%%%%%%%%%%%%%%%%%%%%%%%
%%%%%%%%%%%%%%%%%%%%%%%%%%%%%%%%%%%%%%%%%%%%%%%%%%%%%%%%%%%%%%%%%%%%%%%%%%%%%%
\maketitle
%%%%%%%%%%%%%%%%%%%%%%%%%%%%%%%%%%%%%%%%%%%%%%%%%%%%%%%%%%%%%%%%%%%%%%%%%%%%%%
%%%%%%%%%%%%%%%%%%%%%%%%%%%%%%%%%%%%%%%%%%%%%%%%%%%%%%%%%%%%%%%%%%%%%%%%%%%%%%
\begin{abstract}
 In this paper,  we consider product-type  operators $T^m_{u,v,\varphi}$ from minimal M\"{o}bius invariant spaces into Zygmund type spaces.  So some characterizations for   boundedness and essential norm of these operators are obtained. As a result some conditions for the compactness will be given.
\end{abstract}
%%%%%%%%%%%%%%%%%%%%%%%%%%%%%%%%%%%%%%%%%%%%%%%%%%%%%%%%%%%%%%%%%%%%%%%%%%%%%%
%%%%%%%%%%%%%%%%%%%%%%%%%%%%%%%%%%%%%%%%%%%%%%%%%%%%%%%%%%%%%%%%%%%%%%%%%%%%%%
\section{Introduction}
By $\mathbb{D}$ be the open unit disc in the complex plane $\mathbb{C}$, $H(\mathbb{D})$ is denoted the space of all analytic functions on $\mathbb{D}$.
The classic {Zygmund space} $\mathcal{Z}$ consists of all functions $f \in H(\mathbb{D})$ which are continuous on the closed unit ball
$ \overline{\mathbb{D}}$ and
$$\sup \frac{|f(e^{i(\theta +h)}) + f(e^{i(\theta -h)}) - 2 f (e^{i \theta})|}{h} < \infty,$$
where the supremum is taken over all $\theta \in \mathbb{R}$ and $h> 0$.
By \cite[Theorem 5.3]{peter}, an analytic function $f$ belongs to $\mathcal{Z}$ if and only if $\sup_{z\in \mathbb{D}} (1-|z|^2) |f''(z)| < \infty$.
Motivated by this, for each  $\alpha>0$, the {Zygmund type space} $\mathcal{Z}_{\alpha}$ is defined to be the space of all functions $f \in H(\mathbb{D})$ for which
$$\|f\|_{s\mathcal{Z}_{\alpha}} = \sup_{z\in \mathbb{D}} (1-|z|^2)^{\alpha} |f''(z)| < \infty.$$
The space $\mathcal{Z}_{\alpha}$ is a Banach space equipped with the norm
$$\|f\|_{\mathcal{Z}_{\alpha}} = |f(0)| + |f'(0)| + \|f\|_{s\mathcal{Z}_{\alpha}},$$
for each $f\in \mathcal{Z}_{\alpha}$.

Let $Aut(\mathbb{D})$ be the group of all conformal  automorphisms of $\mathbb{D}$ which is also called M\"{o}bius group. It is well-known that
each element of  $Aut(\mathbb{D})$ is of the form
$$ e^{i \theta} \sigma_a (z)= e^{i \theta} \frac{a-z}{1- \overline{a}z}, \ \ \ \ a,z \in \mathbb{D}, \ \ \ \ \theta \in \mathbb{R}.  $$
Let $X$ be a linear space of analytic functions on $\mathbb{D}$, which is complete. $X$ is called M\"{o}bius invariant if for each function $f$ in $X$ and each element $\psi$ in $Aut(\mathbb{D})$, the composition function $f \circ \psi$ also lies in $X$ and satisfies that
$||f \circ \psi||_X = ||f||_{X}$. For example, the space $H^{\infty}$ of all bounded analytic functions is M\"{o}bius invariant.
Also the Besov space $B_p(1 < p < \infty)$, is M\"{o}bius invariant which is the space of all $f \in H(\mathbb{D})$ such that
\begin{equation}\label{e10}
  \int_{\mathbb{D}} |f'(z)|^p (1-|z|^2)^{p-2} dA(z) < \infty.
\end{equation}
If $p= 2$, we have the well-known Dirichlet space. For $p=\infty$, $B_{\infty} = \mathcal{B}$ the classic Bloch space. The space
$B_1$ which is called the minimal M\"{o}bius invariant is defined separately. The function $f \in H(\mathbb{D})$ belongs to
$B_1$ if and only if it has representation as
$$ f(z) = \sum_{k=1}^{\infty} c_k \sigma_{a_k} \quad \mbox{where} \ \ \ a_k \in \mathbb{D} \ \  \mbox{and }\ \  \sum_{k=1}^{\infty} |c_k| < \infty. $$
The norm on $B_1$ defines as infimum of $\sum_{k=1}^{\infty} |c_k| $ for which the above statement holds. $B_1$ is contained in any
M\"{o}bius invariant space and it has been proved that is the set of all analytic functions $f$ on $\mathbb{D}$ such that
$f''$ lies in $L^1 (\mathbb{D}, dA)$. Also there exist positive constants $C_1$ and $C_2$ such that
\begin{equation}\label{e11}
  C_1 ||f||_{B_1} \leq |f(0)| + |f'(0)| + \int_{\mathbb{D}} |f''(z)| dA(z) \leq C_2 ||f||_{B_1}.
\end{equation}
Let $u, v, \varphi \in H(\mathbb{D})$ and  $\varphi: \mathbb{D} \rightarrow \mathbb{D}$. The Stevi\'c-Sharma type operator is defined as follows
\begin{align*}
T_{u, v, \varphi}f(z) = u(z) f{(\varphi(z))}+ v(z) f'(\varphi(z)) ,\quad f\in H(\mathbb{D} ), \quad z\in \mathbb{D}.
\end{align*}
Indeed $T_{u, v, \varphi}= uC_\varphi+vC_\varphi D$ where $D$ is the differentiation operator and $C_\varphi$ is composition operator.
More information about this operator can be found in \cite{ly, s1, s2}.

The generalized  Stevi\'c-Sharma type operator $T^m_{u,v,\varphi}$ is defined by the second auther of this paper and et al. in \cite{eym} as follows
 \begin{align*}
T^m_{u, v, \varphi}f(z)
= (u C_\varphi f)(z) + (D^m_{\varphi,v} f)(z) = u(z) f(\varphi(z)) + v(z) f^{(m)}(\varphi(z)),
\end{align*}
where $m\in\mathbb{N}$ and $D^m_{\varphi,u}$ is
the  generalized weighted  composition operator. When $v=0$, then $T^m_{u, 0, \varphi} = u C_\varphi $ is the well-known weighted composition operator. If $u = 0,$ then  $T^m_{0, v, \varphi} = D^m_{\varphi,v}$ and for $m=1$,   $T^m_{u, v, \varphi} = T_{u, v, \varphi}$   is Stevi\'c-Sharma type operator.
$T^m_{u, v, \varphi}$ also include other operators as well as product type operators which are studied in several papers in recent years.
The results of the papers can be stated to many operators and obtained the results of the papers published before.

For Banach spaces X and Y and a continuous linear operator $T: X\rightarrow Y$, the essential norm is the distance of $T$ from the space of all compact operators, that is
\[
\|T\|_{e} = \inf\{ \| T-K\| :  K: X\rightarrow Y\hspace{0.1cm}  \text{is compact}\}.
\]
$T$ is compact if and only if $\|T\|_e =0$.

In this paper, we study  the operator-theoretic properties in minimal M\"{o}bius invariant space. 
In section 2, we first bring some lemmas on the space $B_1$ and then  obtain some characterizations for   boundedness of operator $T^m_{u,v,\varphi}: B_1\rightarrow \mathcal{Z}_{\alpha}$. In section 3, some estimations for the essential norm of these operators are given. As a result, some new criteria for the compactness  of  $T^m_{u,v,\varphi}$ are presented.

By $A\succeq B$ we mean there exists a constant $C$ such that $A\geq CB$ and $A\approx B$ means that $A\succeq B\succeq A$.

\section{Boundedness}
In this section, we  give some necessary and sufficient conditions for the  generalized Stevi\'c-Sharma type operators to be bounded.
Firstly, we state some lemmas which are needed for proving the main results.

According to the definition of the norm in minimal  M\"{o}bius invariant space, for each $f \in B_1$, $||f||_{\infty} \leq ||f||_{B_1}$. Thus, from
\cite[Proposition 5.1.2]{far11} and \cite[Proposition 8]{far12} we have the following lemma.

\begin{lemma} \label{l10}
Let $n \in \mathbb{N}$. Then there exists a positive constant $C$ such that for each $f \in B_1$
$$ (1-|z|^2)^n |f^{(n)} (z)| \leq C ||f||_{B_1}. $$
\end{lemma}
As a similar proof in Lemma 2.5 of \cite{{xea1}} we get the following lemma.
\begin{lemma}  \label{l11}
Let
$$ f_{j,a}(z)= \left ( \frac{1-|a|^2}{1-\overline{a}z} \right)^j, \ \ \ \ j \in \mathbb{N}, \ a \in \mathbb{D}. $$
Then
$$ f_{j,a} \in B_1, \ \ \ \ \sup_{a \in \mathbb{D}} ||f_{j,a}||_{B_1} < \infty. $$
Moreover, $\{ f_{j,a} \} \rightarrow 0$ uniformly on compact subsets of $\mathbb{D}$ as $|a| \rightarrow 1$.
\end{lemma}
The proof of the following lemmas are similar to the proof of  Lemmas 2.6  and 2.5 \cite{eym}, so they are omitted.
\begin{lemma} \label{l12}
For any $m \in \mathbb{N}-\{1,2\}$, $0 \not = a \in \mathbb{D}$ and
$i,k \in \{ 0,1, 2, m, m+1, m+2 \}$, there exits a function $g_{i,a} \in B_1$ such that
$$ g_{i,a}^{(k)} (a) = \frac{\delta_{ik} \overline{a}^{k}}{(1-|a|^2)^k}, $$
where $\delta_{ik}$ is Kronecker delta. For each $i\in\{0,1,2\}$  and $i\in\{m,m+1,m+2\}$ respectively
\[
 g_{i,a}(z) = \sum_{j=1}^{3}  c_j^i f_{j,a} (z), \qquad\quad g_{i,a}(z) = \sum_{j=m+1}^{m+3}  c_j^i f_{j,a} (z)
\]
 where $c_j^i$ is independent of $a$.
\end{lemma}
\begin{lemma} \label{m123}
Let $m=1$  or $2$, $0 \not = a \in \mathbb{D}$ and
$i,k \in \{ 0,1, \cdots, m+2 \}$, there exists a function $g_{i,a} \in B_1$ such that
$$ g_{i,a}^{(k)} (a) = \frac{\delta_{ik} \overline{a}^{k}}{(1-|a|^2)^k}. $$
\end{lemma}

Let $f \in B_1$. Then
\begin{align*}
 || T^m_{u, v, \varphi} f ||_{\mathcal{Z}_{\alpha}} =  |T^m_{u, v, \varphi} f (0)| + |(T^m_{u, v, \varphi} f)' (0)| + \sup_{z \in \mathbb{D}}(1-|z|^2)^{\alpha} | (T^m_{u, v, \varphi} f)'' (z) |.
\end{align*}
We compute the above sentences separately. We have
\begin{align}\label{e12}
& (T^m_{u, v, \varphi} f)' (0) =\\
 &u'(0) f(\varphi(0)) + u(0) \varphi'(0) f'(\varphi(0)) + v'(0) f^{(m)} (\varphi(0)) + v (0) \varphi'(0) f^{(m+1)}(\varphi(0)).\nonumber
\end{align}
And
\begin{align*}
  (T^m_{u, v, \varphi} f)'' (z) = & \sum_{i=0}^2 \big(I_i(z) f^{(i)}(\varphi(z))+ I_{i+m}(z) f^{(i+m)}(\varphi(z))\big),
\end{align*}
where,
$$ I_0 = u'' \ \ \ \ I_1 = 2 u'\varphi' + u \varphi'', \ \ \ \ I_2 =u \varphi'^2 $$
$$ I_m = v'' \ \ \ \ I_{m+1} =  2 v'\varphi' + v \varphi'' \ \ \ \ I_{m+2} = v \varphi'^2 $$
%*********************************************************
\begin{theorem}\label{th10}
Let $\alpha>0$, $u, v, \varphi \in H(\mathbb{D})$,  $\varphi: \mathbb{D} \rightarrow \mathbb{D}$ and $m > 2$ be an integer. Then the following conditions are equivalent:
\begin{enumerate}
\item [(i)]
The operator $T^m_{u, v, \varphi} : B_1 \rightarrow \mathcal{Z}_{\alpha} $ is  bounded.
 \item [(ii)] For each $ j \in \{0,1,2,m,m+1, m+2 \}=\mathfrak{Q}$
 \[
 \max\{\sup_{a \in \mathbb{D}} ||T_{u,v,\varphi}^m f_{j,a}||_{\mathcal{Z}_{\alpha}},\ \  \sup_{z \in \mathbb{D}} (1-|z|^2)^{\alpha} |I_j (z)| \}< \infty.
 \]
\item [(iii)]
$$ \sup_{z \in \mathbb{D}} \frac{(1-|z|^2)^{\alpha} |I_k (z)|}{(1-|\varphi(z)|^2)^{k}} < \infty, \ \ \ \ k \in\mathfrak{Q}. $$
 \end{enumerate}
\end{theorem}
\begin{proof}
$(iii)\Rightarrow (i)$  Suppose that $f \in B_1$. By using Lemma \ref{l10}
\begin{align*}
 (1-|z|^2)^{\alpha} |(T^m_{u, v, \varphi} f)'' (z)| = &  (1-|z|^2)^{\alpha} |\sum_{k \in\mathfrak{Q}}  I_k (z) f^{(k)}(\varphi(z)) | \\
 \leq &\sum_{k \in\mathfrak{Q}}  |I_k (z)|  (1-|z|^2)^{\alpha} | f^{(k)}(\varphi(z))| \\
 \leq  & C \sum_{k \in\mathfrak{Q}} \frac{(1-|z|^2)^{\alpha}|I_k (z) |}{(1-|\varphi(z)|^2)^{k}} ||f||_{B_1}.
\end{align*}
Also using the fact that $||f||_{\infty} \leq ||f||_{B_1}$ and Lemma \ref{l10}, we have
\begin{align*}
 |T^m_{u, v, \varphi} f (0)| \leq &  |u(0) f(\varphi(0))| + |v(0) f^{(m)} (\varphi(0))| \\
 \leq & |u(0)| ||f||_{B_1} +  C \frac{|v(0)|}{(1-|\varphi(0)|^2)^{m}} ||f||_{B_1}
\end{align*}
and
\begin{align*}
|(T^m_{u,v,\varphi}f)' (0)|
\leq C\left (|u'(0)|+\frac{|u(0) \varphi'(0)|}{1-|\varphi(0)|^2} + \frac{|v'(0)|}{(1-|\varphi(0)|^2)^{m}} + \frac{|v (0)\varphi'(0)|}{(1-|\varphi(0)|^2)^{m+1}} \right) ||f||_{B_1}.
\end{align*}
Therefore $T^m_{u, v, \varphi} : B_1 \rightarrow \mathcal{Z}_{\alpha} $ is  bounded.

$(i)\Rightarrow (ii)$ Suppose that $T^m_{u, v, \varphi} : B_1 \rightarrow \mathcal{Z}_{\alpha} $ be a  bounded operator. Lemma \ref{l11} implies that $||f_{j,a}||_{B_1} < \infty$. So
$$ ||T_{u,v,\varphi}^m f_{j,a}||_{\mathcal{Z}_{\alpha}} \leq
|| T_{u,v,\varphi}^m  || ||f_{j,a}||_{B_1} < \infty. $$
Then
$$ \sup_{a \in \mathbb{D}} ||T_{u,v,\varphi}^m f_{j,a}||_{\mathcal{Z}_{\alpha}} \leq || T_{u,v,\varphi}^m  || \sup_{a \in \mathbb{D}, j\in\mathfrak{Q}} ||f_{j,a}||_{B_1} < \infty. $$
Define $f_0 (z) =1 \in B_1$. Boundedness of the operator implies that
\begin{align*}
\sup_{z \in \mathbb{D}} (1-|z|^2)^{\alpha} |I_0 (z)| =  \sup_{z \in \mathbb{D}} (1-|z|^2)^{\alpha} |u''(z)| \leq  ||T_{u,v,\varphi}^m f_0||_{\mathcal{Z}_{\alpha}} \leq || T_{u,v,\varphi}^m  || ||f_0||_{B_1} < \infty.
\end{align*}
Take $f_1 (z) = z \in B_1$. Then we have
\begin{align*}
  \sup_{z \in \mathbb{D}} (1-|z|^2)^{\alpha} |u''(z) \varphi(z) + 2 u'(z) \varphi'(z) + u (z) \varphi''(z)| \leq  ||T_{u,v,\varphi}^m f_1||_{\mathcal{Z}_{\alpha}} \leq || T_{u,v,\varphi}^m  || ||f_1||_{B_1} < \infty.
\end{align*}
Using the previous equations, we can get that
\begin{align*}\label{e15}
\sup_{z \in \mathbb{D}} (1-|z|^2)^{\alpha} |I_1 (z)|< \infty.
\end{align*}
Similarly by employing the functions $f_2 (z) = z^2$, $f_m (z) = z^m$, $f_{m+1} (z) = z^{m+1}$ and $f_{m+2} (z) = z^{m+2}$ for operator $ T_{u,v,\varphi}^m$ we get the other part of $(ii)$.

$(ii)\Rightarrow (iii)$ For any $i \in\mathfrak{Q}$ and $a \in \mathbb{D}$, by applying Lemma \ref{l12} we have
\begin{align*}
\frac{(1-|z|^2)^{\alpha} |I_i (a)| |\varphi(a)|^i}{(1-|\varphi(a)|^2)^i}  \leq & (1-|z|^2)^{\alpha} |(T_{u,v,\varphi}^m g_{i, \varphi(a)})''(a)| \\
\leq & \sup_{a \in \mathbb{D}} ||T_{u,v,\varphi}^m g_{i, \varphi(a)}||_{\mathcal{Z}_{\alpha}} \\
\leq &\max\{\sum_{j=1}^{3} c_j^i \sup_{a \in \mathbb{D}} ||T_{u,v,\varphi}^m f_{j,a}||_{\mathcal{Z}_{\alpha}}, \ \
\sum_{j=m+1}^{m+3} c_j^i \sup_{a \in \mathbb{D}} ||T_{u,v,\varphi}^m f_{j,a}||_{\mathcal{Z}_{\alpha}}\}\\ <& \infty.
\end{align*}
So, for any $i \in\mathfrak{Q}$
\begin{equation}\label{e16}
  \sup_{|\varphi(a)| > 1/3} \frac{(1-|z|^2)^{\alpha} |I_i (a)| }{(1-|\varphi(a)|^2)^i} < \infty.
\end{equation}
On the other hand
\begin{equation}\label{e17}
  \sup_{|\varphi(a)| \leq 1/3} \frac{(1-|z|^2)^{\alpha} |I_i (a)| }{(1-|\varphi(a)|^2)^i} \leq C
   \sup_{a \in \mathbb{D}} (1-|z|^2)^{\alpha} |I_i (a)|  < \infty.
\end{equation}
From the  last inequalities, we get the desired result.
 \end{proof}
%***************************************************************
In the special case $m\leq2$, by using Lemma \ref{m123}, we have the following theorems  which is stated without proof.
\begin{theorem}\label{th11}
Let $\alpha>0$,  $u, v, \varphi \in H(\mathbb{D})$ and  $\varphi: \mathbb{D} \rightarrow \mathbb{D}$. Then the following conditions are equivalent:
\begin{enumerate}
\item[(i)]  The operator $T^2_{u, v, \varphi} : B_1 \rightarrow \mathcal{Z}_{\alpha} $ is  bounded.
 \item[(ii)] For $j\in\{1,...,5\}$, $ \sup_{a \in \mathbb{D}} ||T_{u,v,\varphi}^m f_{j,a}||_{\mathcal{Z}_{\alpha}} < \infty$ and
\begin{align*}
&\sup_{z \in \mathbb{D}} (1-|z|^2)^{\alpha}\left |u''(z)| +|( 2 u'\varphi' + u \varphi'')(z)|+|(u \varphi'^2 +v'')(z)|\right)< \infty,\\
&\sup_{z \in \mathbb{D}} (1-|z|^2)^{\alpha}\left |(2 v'\varphi' + v \varphi'')(z)|+|(v \varphi'^2)(z)|\right)<\infty.
\end{align*}
\item[(iii)]
\begin{align*}
&\sup_{z \in \mathbb{D}} (1-|z|^2)^{\alpha}\left( |u''(z)| + \frac{|( 2 u'\varphi' + u \varphi'')(z)|}{(1-|\varphi(z)|^2)}+\frac{|(u \varphi'^2 +v'')(z)|}{(1-|\varphi(z)|^2)^2}\right)< \infty,\\
&\sup_{z \in \mathbb{D}} (1-|z|^2)^{\alpha}\left( \frac{|(2 v'\varphi' + v \varphi'')(z)|}{(1-|\varphi(z)|^2)^3}+\frac{|(v \varphi'^2)(z)|}{(1-|\varphi(z)|^2)^4}\right)<\infty.
\end{align*}
 \end{enumerate}
\end{theorem}

\begin{theorem}\label{th11z}
Let $\alpha>0$,  $u, v, \varphi \in H(\mathbb{D})$ and  $\varphi: \mathbb{D} \rightarrow \mathbb{D}$. Then the following conditions are equivalent:
\begin{enumerate}
\item[(i)]  The operator $T_{u, v, \varphi} : B_1 \rightarrow \mathcal{Z}_{\alpha} $ is  bounded.
 \item[(ii)] For $j\in\{1,...,4\}$, $ \sup_{a \in \mathbb{D}} ||T_{u,v,\varphi}^m f_{j,a}||_{\mathcal{Z}_{\alpha}} < \infty$ and
\begin{align*}
&\sup_{z \in \mathbb{D}} (1-|z|^2)^{\alpha}\left |u''(z)| +|( 2 u'\varphi' + u \varphi'' +v'')(z)|\right)< \infty,\\
&\sup_{z \in \mathbb{D}} (1-|z|^2)^{\alpha}\left |(u \varphi'^2 + 2 v'\varphi' + v \varphi'')(z)|+|(v \varphi'^2)(z)|\right)<\infty.
\end{align*}
\item[(iii)]
\begin{align*}
&\sup_{z \in \mathbb{D}} (1-|z|^2)^{\alpha}\left( |u''(z)| + \frac{|( 2 u'\varphi' + u \varphi'' +v'')(z)|}{(1-|\varphi(z)|^2)}\right)< \infty,\\
&\sup_{z \in \mathbb{D}} (1-|z|^2)^{\alpha}\left( \frac{|(u \varphi'^2+2 v'\varphi' + v \varphi'')(z)|}{(1-|\varphi(z)|^2)^2}+\frac{|(v \varphi'^2)(z)|}{(1-|\varphi(z)|^2)^3}\right)<\infty.
\end{align*}
 \end{enumerate}
\end{theorem}

%*********************************************************************
%*********************************************************************
\section{Essential Norm}
In this section, some estimations for the essential norm of the operator $T^m_{u,v,\varphi}$  from  minimal M\"{o}bius invariant spaces into Zygmund type spaces are  given.
\begin{theorem}\label{th15}
Let $u, v, \varphi \in H(\mathbb{D})$,  $\varphi: \mathbb{D} \rightarrow \mathbb{D}$ and $2<m\in\mathbb{N}$. Let the operator $T^m_{u, v, \varphi} : B_1 \rightarrow \mathcal{Z}_{\alpha} $ be  bounded then
$$ ||T^m_{u, v, \varphi}||_e \approx \max \{ E_i \}_{i=1}^6 \approx \max \{ F_k \}_{k \in \{0,1,2,m,m+1, m+2 \}}$$
where,
$$ E_i = \limsup_{|a| \rightarrow 1} ||T_{u,v,\varphi}^m f_{i,a}||_{\mathcal{Z}_{\alpha}} \ \mbox{and}\ \  F_k = \limsup_{|\varphi(z)| \rightarrow 1} \frac{(1-|z|^2)^{\alpha} |I_k (z)|}{(1-|\varphi(z)|^2)^{k}}.$$
\end{theorem}
\begin{proof}
First we prove the lower estimates. Suppose that $K : B_1 \rightarrow \mathcal{Z}_{\alpha} $ be an arbitrary compact operator. Since  $\{ f_{i,a} \} $ is a bounded sequence in $B_1$ and  converges to $0$ uniformly on compact subsets of $\mathbb{D}$ as $|a| \rightarrow 1$, we have
 $\limsup_{|a| \rightarrow 1} ||K f_{i,a}||_{\mathcal{Z}_{\alpha}} =0.$
So
\begin{align*}
\|T^m_{u, v, \varphi}-K\|_{B_1 \rightarrow \mathcal{Z}_{\alpha}} &\succeq
 \limsup_{|a|\rightarrow 1}\|(T^m_{u, v, \varphi}-K)f_{i,a}\|_{\mathcal{Z}_{\alpha}}=E_i.
 \end{align*}
Then
 \begin{align*}
\|T^m_{u, v, \varphi}\|_{e}= \inf _{K}\|T^m_{u, v, \varphi}-K\|_{B_1\rightarrow \mathcal{Z}_{\alpha}} \succeq
\max\{E_i\}_{i=1}^{6}.
 \end{align*}
For the other part, let $\{z_j\}_{j\in\mathbb{N}}$ be a sequence in $\mathbb{D}$ such that $|\varphi(z_j)|\rightarrow 1$ as $j\rightarrow \infty$. Since $T^m_{u, v, \varphi} : B_1 \rightarrow \mathcal{Z}_{\alpha}$ is bounded, using Lemmas $\ref{l11}$ and $\ref{l12}$   for any compact operator $K: B_1 \rightarrow \mathcal{Z}_{\alpha}$ and $i \in \{0,1,2,m,m+1, m+2 \}$,   we obtain
\begin{align*}
  \| T^m_{u, v, \varphi}-K\|_{B_1 \rightarrow \mathcal{Z}_{\alpha}} &\succeq
 \limsup_{j\rightarrow \infty}\| T^m_{u, v, \varphi}(g_{i, \varphi(z_j)})\|_{\mathcal{Z}_{\alpha}}-
 \limsup_{j\rightarrow \infty}\| K(g_{i, \varphi(z_j)})\|_{\mathcal{Z}_{\alpha}}\nonumber\\
 &\succeq  \limsup_{j\rightarrow \infty}\frac{(1-|z_j|^2)^{\alpha} |\varphi(z_j)|{i} |I_i (z_j)|}{ (1- | \varphi(z_j)|^{2})^{i}}=F_i.
 \end{align*}
 Therefore,
\begin{align*}
\|T^m_{u, v, \varphi}\|_{e}= \inf _{K}\|T^m_{u, v, \varphi}-K\|_{B_1\rightarrow \mathcal{Z}_{\alpha}} \succeq
\max\{F_i\}.
 \end{align*}
Now we prove the upper estimates. Consider the operators $K_r$ on $B_1$, $K_r f(z) = f_r(z) = f(rz)$, where $0 < r <1$.  $K_r$ is a compact operator and $||K_r|| \leq 1$. Let $\{r_j\}\subset (0,1)$ be a sequence such that $r_j\rightarrow 1$ as $j\rightarrow \infty$.
For any positive integer $j$, the operator $T^m_{u,v,\varphi}K_{r_j} : B_1\rightarrow \mathcal{Z}_{\alpha}$ is compact. Thus
\begin{align}\label{e20}
\|T^m_{u,v,\varphi}\|_{e} \leq \limsup_{j\rightarrow \infty}\| T^m_{u,v,\varphi} - T^m_{u,v,\varphi}K_{r_j}\|.
 \end{align}
So it will be sufficient to prove that
\begin{align*}
\limsup_{j\rightarrow \infty}\| T^m_{u,v,\varphi} - T^m_{u,v,\varphi}K_{r_j}\|\preceq \min\{ \max\{E_i\}, \max\{F_i\} \}.
\end{align*}
For any $f\in B_1$ such that $\|f\|_{B_1}\leq 1$,
\begin{align*}
 ||( T^m_{u, v, \varphi} -T^m_{u, v, \varphi} K_{r_j}) f ||_{\mathcal{Z}_{\alpha}} =  A_1 + A_2 + A_3
\end{align*}
where
$$ A_1 =: \left|T^m_{u, v, \varphi} f (0) - T^m_{u, v, \varphi} f_{r_j}(0)\right| = \left|u(0) (f-f_{r_j})(\varphi(0)) + v(0) (f-f_{r_j})^{(m)}(\varphi(0))\right|$$
\begin{align*}
 A_2=: & \left|(T^m_{u, v, \varphi} f - T^m_{u, v, \varphi}f_{r_j})' (0)\right|
  =| u'(0)(f-f_{r_j})(\varphi(0)) + u(0) \varphi'(0) (f-f_{r_j})'(\varphi(0)) \\&+ v'(0) (f-f_{r_j})^{(m)}(\varphi(0)) + v(0) \varphi'(0) (f-f_{r_j})^{(m+1)}(\varphi(0)) |
\end{align*}
\begin{align*}
A_3 =: & \sup_{z \in \mathbb{D}}(1-|z|^2)^{\alpha} \left|(T^m_{u, v, \varphi} f - T^m_{u, v, \varphi} f_{r_j})'' (z)\right| \\
= & \sup_{z \in \mathbb{D}} (1-|z|^2)^{\alpha} \left|\sum_{k \in \{0,1,2,m,m+1, m+2 \}}  I_k (z) (f-f_{r_j})^{(k)}(\varphi(z))\right| \\
\leq &  \sup_{|\varphi(z)| \leq r_N} (1-|z|^2)^{\alpha} \sum_{k \in \{0,1,2,m,m+1, m+2 \}}  \left|I_k (z) (f-f_{r_j})^{(k)}(\varphi(z)) \right| \\
& + \sup_{|\varphi(z)| > r_N} (1-|z|^2)^{\alpha} \sum_{k \in \{0,1,2,m,m+1, m+2 \}}  \left|I_k (z) (f-f_{r_j})^{(k)}(\varphi(z)) \right| \\
=: & A_4  + A_5
\end{align*}
Since $(f-f_{r_j})^{(i)} \rightarrow 0$ uniformly on compact subsets of ${\mathbb{D}}$ as $j\rightarrow \infty$, for any nonnegative integer $i$, then using Theorem \ref{th10}, we get
$$ \limsup_{j \rightarrow \infty} A_1 =\limsup_{j \rightarrow \infty} A_2 =\limsup_{j \rightarrow \infty} A_4 =0.  $$
About $A_5$, we get
\begin{align*}
A_5 \leq  &    \sum_{k \in \{0,1,2,m,m+1, m+2 \}} \sup_{|\varphi(z)| > r_N} (1-|z|^2)^{\alpha} |I_k (z)| | f^{(k)}(\varphi(z)) | \\
& + \sum_{k \in \{0,1,2,m,m+1, m+2 \}} \sup_{|\varphi(z)| > r_N} (1-|z|^2)^{\alpha} |I_k (z)| | r_j^k f^{(k)}(r_j \varphi(z)) | \\
= & \sum_{k \in \{0,1,2,m,m+1, m+2 \}} A_{k,6} + \sum_{k \in \{0,1,2,m,m+1, m+2 \}} A_{k,7}.
\end{align*}
For $A_{k,6}$, using Lemmas \ref{l10}, \ref{l11} and \ref{l12}, we obtain
\begin{align*}
A_{k,6} = &  \sup_{|\varphi(z)| > r_N}  \frac{(1-|\varphi(z)|^2)^k f^{(k)}(\varphi(z))}{|\varphi(z)|^k} \frac{(1-|z|^2)^{\alpha}|I_K(z)| |\varphi(z)|^k}{1-|\varphi(z)|^2)^k}  \\
\leq & ||f||_{B_1} \sup_{|\varphi(z)| > r_N} ||( T^m_{u, v, \varphi} g_{k, \varphi(z)}||_{\mathcal{Z}_{\alpha}}  \\
\preceq & \sum_{j=1}^6 c_j^k \sup_{|a| > r_N} ||T^m_{u, v, \varphi} f_{j,a}||_{\mathcal{Z}_{\alpha}}
\end{align*}
where $k \in \{0,1,2,m,m+1, m+2 \}$.
As $N \rightarrow \infty$,
$$ \limsup_{j \rightarrow \infty} A_{k,6} \preceq \sum_{j=1}^6  \limsup_{|a| \rightarrow 1} || T^m_{u, v, \varphi} f_{j,a}||_{\mathcal{Z}_{\alpha}} \preceq \max \{ E_j \}_{j=1}^6.  $$
Also for $A_{k,6}$ we can write
\begin{align*}
A_{k,6} = &  \sup_{|\varphi(z)| > r_N}  (1-|\varphi(z)|^2)^k f^{(k)}(\varphi(z))  \frac{(1-|z|^2)^{\alpha}|I_K(z)|}{1-|\varphi(z)|^2)^k}  \\
\preceq & ||f||_{B_1} \sup_{|\varphi(z)| > r_N} \frac{(1-|z|^2)^{\alpha}|I_K(z)|}{(1-|\varphi(z)|^2)^k}
\end{align*}
which can be deduced that
$$\limsup_{j \rightarrow \infty} A_{k,6} \preceq  \limsup_{|\varphi(z)| \rightarrow 1}  \frac{(1-|z|^2)^{\alpha}|I_k(z)|}{(1-|\varphi(z)|^2)^k} \leq
\max \{ F_k \}_1^6. $$
A similar argument can de done for $A_{k,7}$. Thus we prove that
$$ \sup_{||f||_{B_1} \leq 1} ||( T^m_{u, v, \varphi} -T^m_{u, v, \varphi} K_{r_j}) f ||_{\mathcal{Z}_{\alpha}} \preceq \max \{ E_j \}_{1}^6 $$
and
$$ \sup_{||f||_{B_1} \leq 1} ||( T^m_{u, v, \varphi} -T^m_{u, v, \varphi} K_{r_j}) f ||_{\mathcal{Z}_{\alpha}} \preceq \max \{ F_k \}_1^6. $$
Finally we have
\begin{align*}
\limsup_{j\rightarrow \infty}\| T^m_{u,v,\varphi} - T^m_{u,v,\varphi}K_{r_j}\|\preceq \min\{ \max\{E_i\}_1^6, \max\{F_k\}_1^6 \}.
\end{align*}
\end{proof}
%*******************************************************************
In the case $m\leq 2$, the similar result can be stated using Theorems \ref{th11} and  \ref{th11z}.
\begin{theorem}\label{th112}
Let $\alpha>0$,  $u, v, \varphi \in H(\mathbb{D})$ and  $\varphi: \mathbb{D} \rightarrow \mathbb{D}$ and  the operator $T^2_{u, v, \varphi} : B_1 \rightarrow \mathcal{Z}_{\alpha} $ be  bounded. Then
\begin{align*}
&||T^2_{u, v, \varphi}||_e \approx \max \{\limsup_{|a|\rightarrow 1} ||T_{u,v,\varphi}^2 f_{j,a}||_{\mathcal{Z}_{\alpha}} \}_{1}^5 \approx\\
&\limsup_{|\varphi(z)|\rightarrow 1} (1-|z|^2)^{\alpha}\left( |u''(z)| + \frac{|( 2 u'\varphi' + u \varphi'')(z)|}{(1-|\varphi(z)|^2)}+\frac{|(u \varphi'^2 +v'')(z)|}{(1-|\varphi(z)|^2)^2}\right)+\\
&\limsup_{|\varphi(z)|\rightarrow 1} (1-|z|^2)^{\alpha}\left( \frac{|(2 v'\varphi' + v \varphi'')(z)|}{(1-|\varphi(z)|^2)^3}+\frac{|(v \varphi'^2)(z)|}{(1-|\varphi(z)|^2)^4}\right).
\end{align*}
\end{theorem}
\begin{theorem}\label{th112z}
Let $\alpha>0$,  $u, v, \varphi \in H(\mathbb{D})$ and  $\varphi: \mathbb{D} \rightarrow \mathbb{D}$ and  the operator $T_{u, v, \varphi} : B_1 \rightarrow \mathcal{Z}_{\alpha} $ be  bounded. Then
\begin{align*}
&||T_{u, v, \varphi}||_e \approx \max \{\limsup_{|a|\rightarrow 1} ||T_{u,v,\varphi} f_{j,a}||_{\mathcal{Z}_{\alpha}} \}_{1}^4 \approx\\
&\limsup_{|\varphi(z)|\rightarrow 1} \left( |u''(z)| + \frac{|( 2 u'\varphi' + u \varphi'' +v'')(z)|}{(1-|\varphi(z)|^2)}\right)+\\
&\limsup_{|\varphi(z)|\rightarrow 1}  (1-|z|^2)^{\alpha}\left( \frac{|(u \varphi'^2+2 v'\varphi' + v \varphi'')(z)|}{(1-|\varphi(z)|^2)^2}+\frac{|(v \varphi'^2)(z)|}{(1-|\varphi(z)|^2)^3}\right).
\end{align*}
\end{theorem}
By using Theorem \ref{th15}, we have the following Corollary.
\begin{corollary}\label{th16}
Let $u, v, \varphi \in H(\mathbb{D})$,  $\varphi: \mathbb{D} \rightarrow \mathbb{D}$ and $2<m\in\mathbb{N}$. Let  operator $T^m_{u, v, \varphi} : B_1 \rightarrow \mathcal{Z}_{\alpha} $ be  bounded. Then the following conditions are equivalent:
\begin{enumerate}
\item [(i)]
The operator $T^m_{u, v, \varphi} : B_1 \rightarrow \mathcal{Z}_{\alpha} $ is  compact.
 \item [(ii)]
 $$ \limsup_{|a| \rightarrow 1} ||T_{u,v,\varphi}^m f_{i,a}||_{\mathcal{Z}_{\alpha}} =0, \ \ \ \ i = 1, 2,\cdots, 6 $$
\item [(iii)]
$$ \limsup_{|\varphi(z)| \rightarrow 1}  \frac{(1-|z|^2)^{\alpha} |I_k (z)|}{(1-|\varphi(z)|^2)^{k}} =0, \ \ \ \ k \in \{0,1,2,m,m+1, m+2 \}. $$
 \end{enumerate}
\end{corollary}
From Theorems \ref{th112} and \ref{th112z}, we obtian similar results for compactness of operator $T^2_{u, v, \varphi} : B_1 \rightarrow \mathcal{Z}_{\alpha}$
and $T_{u, v, \varphi} : B_1 \rightarrow \mathcal{Z}_{\alpha}$ respectively.
\begin{rem}
By taking $u=0(v=0)$, we can get the results of the paper for generalized (weighted) composition operators.
\end{rem}

%%%%%%%%%%%%%%%%%%%%%%%%%%%%%%%%%%%%%%%%%%%%%%%%%%%%%%%%%%%%%%%%%%%%%%%%%%%%%%
%%%%%%%%%%%%%%%%%%%%%%%%%%%%%%%%%%%%%%%%%%%%%%%%%%%%%%%%%%%%%%%%%%%%%%%%%%%%%%

\end{document}